\documentclass[11pt]{article}
\usepackage[margin=1in]{geometry} 
\usepackage{amssymb,amsfonts,amsmath,bbm,mathrsfs,stmaryrd}
\usepackage{xcolor}
\usepackage{url}

\usepackage{enumerate}

\usepackage[colorlinks,
             linkcolor=black!75!red,
             citecolor=blue,
             pdftitle={},
             pdfauthor={Alexandru Chirvasitu, Michael Brannan},
             pdfproducer={pdfLaTeX},
             pdfpagemode=None,
             bookmarksopen=true
             bookmarksnumbered=true]{hyperref}
            
\usepackage{tikz}
\usetikzlibrary{arrows,calc,decorations.pathreplacing,decorations.markings,intersections,shapes.geometric,through,fit,shapes.symbols,positioning,decorations.pathmorphing}

\usepackage{braket}

\usepackage[amsmath,thmmarks,hyperref]{ntheorem}
\usepackage{cleveref}

\creflabelformat{enumi}{#2(#1)#3}

\crefname{section}{Section}{Sections}
\crefformat{section}{#2Section~#1#3} 
\Crefformat{section}{#2Section~#1#3} 

\crefname{subsection}{\S}{\S\S}
\crefformat{subsection}{#2\S#1#3} 
\Crefformat{subsection}{#2\S#1#3} 

\theoremstyle{plain}

\newtheorem{lemma}{Lemma}[section]
\newtheorem{proposition}[lemma]{Proposition}
\newtheorem{corollary}[lemma]{Corollary}
\newtheorem{theorem}[lemma]{Theorem}

\theoremstyle{nonumberplain}

\theoremstyle{plain}
\theorembodyfont{\upshape}
\theoremsymbol{\ensuremath{\blacklozenge}}

\newtheorem{definition}[lemma]{Definition}
\newtheorem{example}[lemma]{Example}
\newtheorem{remark}[lemma]{Remark}

\crefname{definition}{definition}{definitions}
\crefformat{definition}{#2definition~#1#3} 
\Crefformat{definition}{#2Definition~#1#3} 

\crefname{ex}{example}{examples}
\crefformat{example}{#2example~#1#3} 
\Crefformat{example}{#2Example~#1#3} 

\crefname{remark}{remark}{remarks}
\crefformat{remark}{#2remark~#1#3} 
\Crefformat{remark}{#2Remark~#1#3} 

\crefname{convention}{convention}{conventions}
\crefformat{convention}{#2convention~#1#3} 
\Crefformat{convention}{#2Convention~#1#3}

\crefname{lemma}{lemma}{lemmas}
\crefformat{lemma}{#2lemma~#1#3} 
\Crefformat{lemma}{#2Lemma~#1#3} 

\crefname{proposition}{proposition}{propositions}
\crefformat{proposition}{#2proposition~#1#3} 
\Crefformat{proposition}{#2Proposition~#1#3} 

\crefname{corollary}{corollary}{corollaries}
\crefformat{corollary}{#2corollary~#1#3} 
\Crefformat{corollary}{#2Corollary~#1#3} 

\crefname{theorem}{theorem}{theorems}
\crefformat{theorem}{#2theorem~#1#3} 
\Crefformat{theorem}{#2Theorem~#1#3} 

\crefname{assumption}{assumption}{Assumptions}
\crefformat{assumption}{#2assumption~#1#3} 
\Crefformat{assumption}{#2Assumption~#1#3} 

\crefname{equation}{}{}
\crefformat{equation}{(#2#1#3)} 
\Crefformat{equation}{(#2#1#3)}

\theoremstyle{nonumberplain}
\theoremsymbol{\ensuremath{\blacksquare}}

\newtheorem{proof}{Proof}
\newcommand\pf[1]{\newtheorem{#1}{Proof of \Cref{#1}}}

\newcommand\bC{{\mathbb C}}

\newcommand\bN{{\mathbb N}}

\newcommand\bT{{\mathbb T}}

\newcommand\bZ{{\mathbb Z}}

\newcommand\cA{{\mathcal A}}

\newcommand\cC{{\mathcal C}}

\newcommand\cM{{\mathcal M}}

\newcommand\cO{{\mathcal O}}

\title{Topological generation results for free unitary and orthogonal groups}
\author{Alexandru Chirvasitu}

\begin{document}

\date{}

\newcommand{\Addresses}{{
  \bigskip
  \footnotesize

  \textsc{Department of Mathematics, University at Buffalo, Buffalo,
    NY 14260-2900, USA}\par\nopagebreak \textit{E-mail address}:
  \texttt{achirvas@buffalo.edu}

}}

\maketitle

\begin{abstract}
  We show that for every $N\ge 3$ the free unitary group $U^+_N$ is topologically generated by its classical counterpart $U_N$ and the lower-rank $U^+_{N-1}$. This allows for a uniform inductive proof that a number of finiteness properties, known to hold for all $N\ne 3$, also hold at $N=3$. Specifically, all discrete quantum duals $\widehat{U^+_N}$ and $\widehat{O^+_N}$ are residually finite, and hence also have the Kirchberg factorization property and are hyperlinear. As another consequence, $U^+_N$ are topologically generated by $U_N$ and their maximal tori $\widehat{\mathbb{Z}^{*N}}$ (dual to the free groups on $N$ generators) and similarly, $O^+_N$ are topologically generated by $O_N$ and {\it their} tori $\widehat{\mathbb{Z}_2^{*N}}$.
\end{abstract}


\vspace{.5cm}



\section*{Introduction}

The compact quantum group literature has recently seen considerable interest in the notion of {\it topological generation} (e.g. \cite{bcv,ban-ax,ban-hl,ban-tor}). The term was coined in \cite{bcv}, and the concept naturally extends its classical counterpart, applicable to ordinary compact groups:

Let $H,K\le G$ be two closed subgroups of a compact group. Let also $\cO(-)$ denote the algebra of {\it representative functions} on a compact group (i.e. the algebra generated by matrix entries of finite-dimensional representations on Hilbert spaces). Then, $G$ is the closure of the subgroup generated by $H$ and $K$, written
\begin{equation*}
  G=\langle H,\ K\rangle,
\end{equation*}
if and only if there is no proper quotient $*$-Hopf algebra $\cO(G)\to \cA$ through which both
\begin{equation*}
  \cO(G)\to \cO(H)\quad\text{and}\quad \cO(G)\to \cO(K)
\end{equation*}
factor.

Formally, the compact quantum groups in this paper are objects dual to the CQG algebras of \cite{dk}: cosemisimple complex Hopf $*$-algebras $\cA$ with positive unital integral $h:\cA\to \bC$. These are to be thought of as algebras $\cO(G)$ of representative functions on the corresponding compact quantum group. Given that classically topological generation can be cast in function-algebra terms, the notion carries over to the quantum setting (this is a paraphrase of \cite[Definition 4]{bcv}):

\begin{definition}\label{def.tg}
  Let $G_i\le G$, $i\in I$ be a family of quantum subgroups of a compact quantum group. The family {\it topologically generates} $G$, written as
  \begin{equation*}
    G=\langle G_i,\ i\in I\rangle
  \end{equation*}
  if the quotients
  \begin{equation*}
    \cO(G)\to \cO(G_i)
  \end{equation*}
  do not factor through any proper quotient Hopf $*$-algebra of $\cO(G)$. 
\end{definition}
Here (and throughout the paper) `{\it quantum subgroups}' means quotient Hopf $*$-algebra. 

The notion of topological generation appears under different terminology in \cite{chi-rfd}. Specifically, \cite[Defintiion 2.15]{chi-rfd} introduces the concept of a {\it jointly full} family of functors $\cC\to \cC_i$. In that language, $G_i$ topologically generate $G$ precisely when the scalar corestriction functors
\begin{equation}\label{eq:10}
  \cM^{\cO(G)}\to \cM^{\cO(G_i)}
\end{equation}
between the respective categories of comodules form a jointly full family.

As yet another characterization, the condition is equivalent to the requirement that the morphism
\begin{equation*}
  \cO(G)\to \prod^{CQG} \cO(G_i)
\end{equation*}
to the product in the category of CQG algebras ({\it not} the product of underlying algebras!) resulting from the maps \Cref{eq:10} is an embedding.

In the language of \cite[Definition 1.12]{bbcw}, \cite[Theorem 3.1]{chi-rfd} shows that the Pontryagin duals to the free unitary and orthogonal groups $U^+_N$ and $O^+_N$, $N\ge 2$ are residually finite provided $N\ne 3$ (see \Cref{se.prel} below for a reminder on these objects). The reason for the gap is that the argument proceeds inductively, using the following topological generation result (\cite[Lemma 3.13 and Remark 3.14]{chi-rfd}):

\begin{theorem}\label{th.ugen}
  For all $N\ge 4$ we have
  \begin{equation*}
    U^+_N = \langle U_N,\ U^+_{N-1}\rangle.
  \end{equation*}
\end{theorem}
Similarly, the case $N=3$ is problematic in \cite{bcv,bw} for reasons ultimately traceable to the same phenomena (see e.g. \cite[Remark 5]{bcv}).

One of the main goals of the present note is to extend \Cref{th.ugen} to $N=3$ in \Cref{pr.gen-rec} below. This will then have a number of consequences:
\begin{itemize}
\item First off, the results of \cite{bcv,chi-rfd,bw} extend to $N=3$.
\item Secondly, we obtain in \Cref{se.appl} results to the effect that for all $N\ge 2$ the free unitary and orthogonal quantum groups $U^+_N$ and $O^+_N$ are topologically generated by their ``maximal tori''.
\end{itemize}

The underlying Hopf algebra $\cO(U^+_N)$ is generated as a $*$-algebra by generators $u_{ij}$, $1\le i,j\le N$. Annihilating the off-diagonal generators
\begin{equation*}
  u_{ij},\ 1\le i\ne j\le N
\end{equation*}
produces a quotient Hopf algebra isomorphic to the group algebra $\cO(\bT^+_N)$ of the free group on $N$ generators (the images of the diagonal $u_{ii}$). This is the non-commutative analogue of the function algebra on the ``standard torus'' embedded diagonally in the unitary group $U_N$, and the second bullet point above is a paraphrase of \Cref{th.u}, that $U^+_N$ is generated by all of the conjugates of $\bT^+_N$ by the elements of the classical subgroup $U_N<U^+_N$.

The orthogonal picture provided by \Cref{th.o} is similar, the analogue of the maximal torus $\bT^+_N$ this time being the Pontryagin dual of the free product $\bZ_2^{*N}$ of $N$ copies of $\bZ_2$.

The study of maximal tori in compact quantum groups $G$ (meaning maximal cocommutative Hopf $*$-algebra quotients of $\cO(G)$) was initiated in \cite{ban-pat} and \cite{ban-tor} treats torus-generation themes similar to some the present paper's contents.

\Cref{se.prel} briefly recalls some of the relevant background.

In \Cref{se.main} we prove \Cref{pr.gen-rec}, extending the topological generation result in \Cref{th.ugen} to the case $N=3$.

Finally, \Cref{se.appl} records the torus-generation consequences alluded to above, in \Cref{th.u,th.o}.

\subsection*{Acknowledgements}

This work was partially supported by NSF grant DMS-1801011.

I am grateful for numerous stimulating conversations with Teodor Banica (who has been very generous in sharing his problems), Michael Brannan and Amaury Freslon.

\section{Preliminaries}\label{se.prel}

We assume some Hopf algebra background; \cite{swe,abe,mon,rad} are all good sources. For the purposes of this paper compact quantum groups appear in their CQG algebra guise, as in \cite[Definition 2.2]{dk}. An equivalent characterization of those objects reads as follows. 

\begin{definition}\label{def.cqg}
  A {\it CQG algebra} is a complex Hopf $*$-algebra $\cA$ with the following properties:
  \begin{itemize}
  \item $\cA$ is cosemisimple, i.e. its category of (either left or right) comodules is semisimple.
  \item the unique unital left and right integral (\cite[Definition 2.4.4 and Theorem 2.4.6]{mon}) $h:\cA\to \bC$ is positive, in the sense that $h(x^*x)\ge 0$ for all $x\in \cA$.
  \end{itemize}

  The category of {\it compact quantum groups} is the category dual to that of CQG algebras. We write $\cO(G)$ for the Hopf algebra attached to a compact quantum group $G$. 
\end{definition}

We occasionally refer to compact quantum groups as `quantum groups', the phrase being unambiguous throughout this note.

\begin{example}
The preeminent examples in our context are algebras $\cO(G)$ of {\it representative functions} on compact groups $G$. $h$ is then simply integration against the Haar measure on $G$, justifying the terminology of {\it Haar state} for the functional $h$ from \Cref{def.cqg}.   
\end{example}

\begin{example}\label{ex.u}
  For $N\ge 2$ the {\it free unitary group} $U^+_N$ is the compact quantum group whose underlying CQG algebra $\cO(U^+_N)$ is generated by $u_{ij}$, $1\le i,j\le N$ as a $*$-algebra, with relations demanding that both
  \begin{equation*}
    u:=(u_{ij})_{i,j}\quad\text{and}\quad \overline{u}:=(u^*_{ij})_{i,j}\in M_N(\cA)
  \end{equation*}
  are unitary.

  The Hopf algebra structure is given by comultiplication and counit
  \begin{equation*}
    \Delta:u_{ij}\mapsto \sum_k u_{ik}\otimes u_{kj}\quad\text{and}\quad \varepsilon:u_{ij}\mapsto \delta_{ij}
  \end{equation*}
  respectively, where $\delta_{ij}$ denotes the Kronecker delta. 
\end{example}

\begin{example}\label{ex.o}
Let $N\ge 2$ again. The {\it free orthogonal group} $O^+_N$ has associated Hopf algebra $\cO(O^+_N)$ as in \Cref{ex.u}, with the additional constraints that the generators $u_{ij}$ be self-adjoint.  
\end{example}

\Cref{ex.u,ex.o} are central to the discussion below. The objects were introduced in \cite{DaeWan96}, there the respective Hopf algebras were denoted by $A_u(N)$ and $A_o(N)$. Similarly, we have the following construction introduced in \cite{Wan98}.

\begin{example}\label{ex.s}
For $N\ge 2$ the {\it free symmetric group} $S^+_N$ has CQG algebra $\cO(S^+_N)$ generated by the self-adjoint idempotent elements $u_{ij}$ (i.e. projections) such that the sums across the rows and columns of the matrix $u=(u_{ij})_{i,j}$ in $M_N(\cO(S^+_N))$ are all equal to $1\in \cO(S^+_N)$. 
\end{example}

\begin{example}\label{ex.g}
  For every discrete group $\Gamma$ the group algebra $\bC\Gamma$ is a CQG algebra when equipped with the $*$-structure making all $\gamma\in \Gamma$ unitary and with the comultiplication and counit
  \begin{equation*}
    \Delta:\gamma\mapsto \gamma\otimes \gamma\quad\text{and}\quad \varepsilon:\gamma\mapsto 1.
  \end{equation*}

Conceptually, the compact quantum group attached to this CQG algebra should be thought of as the {\it Pontryagin dual} of $\Gamma$, for which reason we write $\bC\Gamma=\cO(\widehat{\Gamma})$. 
\end{example}

Motivated by \Cref{ex.g} we write, as is customary in the field, $\widehat{G}$ for the {\it discrete} quantum group'' dual to $G$. For our purposes, no separate definition of a discrete quantum group is necessary: $\widehat{G}$ is to be thought of simply as a virtual object whose underlying ``group algebra'' $\cO(G)$.

We will also refer frequently to representations of a compact quantum group:

\begin{definition}\label{def.rep}
  Let $G$ be a compact quantum group. Its {\it category of representations} is the category $\cM^{\cO(G)}$ of $\cO(G)$-comodules.
\end{definition}

Just as is the case for an ordinary compact group, representations form a monoidal category. Furthermore, the finite-dimensional representations form a {\it rigid} monoidal category: every object $V$ has a dual $V^*$ and the canonical evaluation map $V^*\otimes V\to \bC$ is a morphism in $\cM^{\cO(G)}$ to the trivial object $\bC$. We refer to \cite[Chapter 2]{egno} for background on the topic.

\section{Free unitary groups in small degree}\label{se.main}

As recalled in the introduction, one of the results of \cite{chi-rfd} is the fact that for all $N\ge 4$ we have
\begin{equation}\label{eq:1}
  U^+_N = \langle U_N,\ U^+_{N-1}\rangle:
\end{equation}
see \cite[Lemma 3.13 and Remark 3.14]{chi-rfd}. 

The first aim of the present note is to observe that in fact the proof of \cite[Lemma 3.13]{chi-rfd} can be slightly altered so as to also allow \Cref{eq:1} for $N=3$:

\begin{proposition}\label{pr.gen-rec}
  For all $N\ge 3$ we have
\begin{equation*}
  U^+_N = \langle U_N,\ U^+_{N-1}\rangle.
\end{equation*}  
\end{proposition}

Let $V$ be the fundamental representation of $U_N^+$. Reprising the notation of \cite{chi-rfd}, for a tuple $(\varepsilon_i)_{i=1}^k$ consisting of blanks and `$*$' we set
\begin{equation*}
  V^{(\varepsilon_i)}:= V^{\varepsilon_1}\otimes\cdots\otimes V^{\varepsilon_k}
\end{equation*}
(i.e. a tensor product of copies of $V$ and its dual $V^*$). 

The representation theory of $U^+_N$ is well understood, having been worked out in \cite{ban-un}. In particular, every $U^+_N$-intertwiner
\begin{equation*}
  V^{(\varepsilon_i)}\to \bC
\end{equation*}
is a linear combination of {\it non-crossing partitions}: evaluation of each copy of $V^*$ against one copy of $V$ in such a manner that strings connecting the $(V^*,V)$ pairs can be drawn in the plane so as not to intersect (and all $V$ and $V^*$ get paired off).

\pf{pr.gen-rec}
\begin{pr.gen-rec}
  As noted above, \cite[Lemma 3.13]{chi-rfd} already takes care of the case $N\ge 4$ so we assume $N=3$. We will examine the proof of that result and amplify it appropriately.

  In the proof in question, one considers a tensor product $V^{(\varepsilon_i)}$ of the fundamental representation of $U_N^+$ and seeks to argue that a linear map $f:V^{(\varepsilon_i)}\to \bC$ is a $U^+_N$-intertwiner provided it is an intertwiner over both $U_N$ and $U^+_{N-1}$.

  In turn, the hypothesis on $f$ is reinterpreted as follows: for every decomposition $V=W\oplus U$ with
  \begin{equation*}
    \dim W=N-1,\quad \dim U=1
  \end{equation*}
  the restriction of $f$ to every summand of the form
  \begin{equation}\label{eq:2}
    Z^{(\varepsilon_i)}\subset V^{(\varepsilon_i)},\quad Z\in\{W,U\}
  \end{equation}
  is a linear combination of non-crossing pairings between $W$, $W^*$ and $U$, $U^*$ tensorands. The conclusion would then have to be that $f$ is a span of non-crossing pairings $V^*\otimes V\to \bC$.

  By moving the complement $U$ of $W$ continuously, we can furthermore assume that $W$ and $U$ are in arbitrary position: $U$ can be contained in $W$ as well as complementary to it. Furthermore, by subtracting an appropriate span of non-crossing pairings from $f$ we can assume that
  \begin{equation*}
    f|_{W^{(\varepsilon_i)}}\cong 0;
  \end{equation*}
  the desired conclusion would then be that $f$ vanishes identically on $V^{(\varepsilon_i)}$.

  We now proceed essentially as in the proof of \cite[Lemma 3.13]{chi-rfd}. Consider the restriction of $f$ to a subspace of the form \Cref{eq:2} for a one-dimensional complement $U$ of $W$. Let $(\omega_j)$ be the sub-tuple of $(\varepsilon_i)$ consisting of those $i$ for which the $i^{th}$ tensorand $Z$ in $V^{(\varepsilon_i)}$ is $W$ (i.e. the $(N-1)$-dimensional subspace rather than the $1$-dimensional one). If $e$ and $e^*$ are mutually dual bases of $U$ and $U^*$ respectively, let
  \begin{equation*}
    \psi: W^{(\delta_j)}\to Z^{(\varepsilon_i)}
  \end{equation*}
  be the map obtained by inserting $e$ and $e^*$ in the spots deleted in passing from $(varepsilon_i)$ to $(\delta_j)_j$ (see \Cref{ex.insert}). Then, the restriction of $f$ to $Z^{(\varepsilon_i)}$ is uniquely determined by $f\circ\psi$, and the latter is a linear combination
  \begin{equation}\label{eq:3}
    \sum_{\pi}c_{\pi}\pi
  \end{equation}
  of non-crossing pairings $\pi$ of $W$ and $W^*$ tensorands in $W^{(\delta_j)}$. Because $f$ is a $U_N$-intertwiner, the same linear combination \Cref{eq:3} is valid for any choice of $W$ and $U$, including, as observed above, the non-complementary case of $U\subseteq W$. But in that case $f=0$ because $f$ vanishes on $W^{(\varepsilon_i)}$, meaning that \Cref{eq:3} is identically zero.

  The linear independence of the non-crossing pairings on an $(N-1)$-dimensional space $W$ for $N-1=2$ now finishes the proof.
\end{pr.gen-rec}

\begin{example}\label{ex.insert}
  If, say,
  \begin{equation*}
    Z^{(\varepsilon_i)} = W\otimes U\otimes U^*\otimes W^*
  \end{equation*}
  then
  \begin{equation*}
    W^{(\delta_j)}=W\otimes W^* \text{ and }\psi:W\otimes W^*\to W\otimes U\otimes U^*\otimes W^*
  \end{equation*}
  is the map
  \begin{equation*}
    W\otimes W^*\ni w\otimes w^* \mapsto w\otimes e\otimes e^*\otimes w^*\in W\otimes U\otimes U^*\otimes W^*.
  \end{equation*}
\end{example}

\subsection{Residual finiteness, hyperlinearity and factorization}

As noted briefly in the introductory discussion above, \Cref{pr.gen-rec} allows us to extend some of the finiteness results in the literature to $U^+_3$ and $O^+_3$. This requires that we recall some terminology. The following notion appears in \cite[Definition 1.12]{bbcw}. 

\begin{definition}\label{def.rfd}
  Let $G$ be a compact quantum group. The corresponding discrete quantum group $\widehat{G}$ is {\it residually finite} if $\cO(G)$ is finitely generated and embeds as a $*$-algebra in a product of matrix algebras.
\end{definition}

With this in hand, we have

\begin{theorem}\label{th.rfd}
  One can simply reprise the proof of \cite[Theorem 3.1]{chi-rfd}, in light of the new information provided by \Cref{pr.gen-rec} at $N=3$.
  
  For all $N\ge 2$ the discrete quantum groups $\widehat{U^+_N}$ and $\widehat{O^+_N}$ are residually finite.
\end{theorem}
\begin{proof}
  \cite[Corollary 2.16]{chi-rfd} says that residual finiteness for discrete quantum groups $\widehat{G}$ is inherited from families $\widehat{G_i}$ if
  \begin{equation*}
    G=\langle G_i\rangle
  \end{equation*}
  is a topologically generating family. \cite[Lemma 3.9]{chi-rfd} proves the claim for $\widehat{U^+_2}$, and since residual finiteness also holds for duals $\widehat{U_N}$ of classical unitary groups the conclusion in the unitary case follows inductively from \Cref{pr.gen-rec}.

  As for the orthogonal claim, it is equivalent to the unitary counterpart by \cite[Proposition 3.8]{chi-rfd}.
\end{proof}

Residual finiteness, in turn, entails other properties of interest in the literature. The {\it Connes embedding problem} (CEP for short) raised in \cite{cep} has driven much of the development in operator algebras. It asks whether every finite von Neumann algebra $N$ with separable predual $N_*$ equipped with a trace $\tau$ embeds in a trace-preserving fashion in an ultrapower $R^{\omega}$ of the hyperfinite $II_1$ factor $R$ with respect to some ultrafilter $\omega$ on $\bN$.

A discrete group is {\it hyperlinear} when its GNS von Neumann algebra with respect to the standard trace satisfies CEP. Motivated by this, we have \cite[\S 3.2]{bcv}.

\begin{definition}\label{def.hyp}
  A discrete quantum group $\widehat{G}$ is {\it hyperlinear} if the GNS von Neumann algebra associated to $\cO(G)$ equipped with the state $h$ satisfies CEP.
\end{definition}

\begin{remark}\label{re.kac}
  Note that the discussion in \Cref{def.hyp} is only meaningful for {\it Kac type} quantum groups, i.e. those for which the Haar state $h$ is tracial: $h(xy)=h(yx)$ for all $x,y\in \cO(G)$.
\end{remark}

The following result extends \cite[Corollary 4.3, Theorem 4.4]{bcv} to $N=3$. 

\begin{corollary}\label{cor.hyp}
The discrete duals $\widehat{U^+_N}$ and $\widehat{O^+_N}$ are hyperlinear for all $N\ge 2$.   
\end{corollary}
\begin{proof}
This follows from \Cref{th.rfd}: as in the case of ordinary discrete groups, residual finiteness entails hyperlinearity (see \Cref{re.rfd-f-cep} below). 
\end{proof}

We next we turn to the Kirchberg factorization property, introduced in \cite{kirch} for discrete groups and studied more generally in the context of discrete {\it quantum} groups in \cite{bw,bbcw} (see \cite[Theorem 28 and Definitions 2.9 and 2.10]{bw}).

\begin{definition}\label{def.fact}
  A discrete quantum group $\widehat{G}$ has the {\it Kirchberg factorization property} if the natural action of the algebraic tensor product $\cO(G)\otimes \cO(G)^{op}$ on the GNS Hilbert space $L^2(G,h)$ with respect to the Haar stat extends to the minimal $C^*$ tensor product of the two enveloping $C^*$-algebras.
\end{definition}

Once more, the bulk of the following result known: all cases $N\ne 3$ are covered by \cite[Theorems 4.3 and 4.4]{bw}.

\begin{corollary}\label{cor.fact}
All $\widehat{U^+_N}$ and $\widehat{O^+_N}$, $N\ge 2$ have the Kirchberg factorization property.   
\end{corollary}
\begin{proof}
  This is again a consequence of \Cref{pr.gen-rec}: by \cite[Theorem 2]{bbcw} residual finiteness implies the Kirchberg factorization property.
\end{proof}

\begin{remark}\label{re.rfd-f-cep}
  The proof of \Cref{cor.fact} also sheds some light on that of \Cref{cor.hyp}: as observed in \cite[Remark 2.11]{bw}, the factorization property implies hyperlinearity. Together with \cite[Theorem 2]{bbcw} this justifies the claim made in the proof of \Cref{cor.hyp} that residual finiteness does too, and shows that the three properties discussed above are ordered by strength as follows:
  \begin{equation*}
    \text{residual finiteness }\Rightarrow\text{ factorization property }\Rightarrow\text{ hyperlinearity}. 
  \end{equation*}  
\end{remark}

\section{An application to generation by tori}\label{se.appl}


Assuming \Cref{pr.gen-rec}, we propose to address the following problem. Denote
\begin{equation*}
  T^+_N=\widehat{\bZ_2^{*N}},\quad \bT^+_N=\widehat{\bZ^{*N}}
\end{equation*}

Then, we have

\begin{theorem}\label{th.u}
  For all $N\ge 2$,
  \begin{equation*}
    U^+_N = \langle U_N, \bT^+_N\rangle.
  \end{equation*}
\end{theorem}
\begin{proof}
  We do this by induction on $N$.
  
  {\bf Induction step: $N\ge 3$.} We know from \Cref{pr.gen-rec} that
  \begin{equation*}
    U^+_N=\langle U_N,U^+_{N-1}\rangle
  \end{equation*}
  and by the induction hypothesis
  \begin{equation*}
    U^+_{N-1} = \langle U_{N-1},\bT^+_{N-1}\rangle. 
  \end{equation*}
The conclusion now follows from $\bT^+_{N-1}<\bT^+_N$. 

{\bf Base case: $N=2$.} According to \cite[Lemme 7]{ban-un} we have a surjection
\begin{equation}\label{eq:4}
  \bT\,\hat *\, SU_2\to U^+_2
\end{equation}
(in the sense that the opposite morphism of Hopf algebras is an inclusion). Since the left hand side is generated by $\bT\,\hat *\,\bT$ and $SU_2$ (by \Cref{th.tg} applied to $G=SU_2$), the right hand side will be generated by the images $\bT^+_2$ and $SU_2\subset U_2$ of these two quantum groups through \Cref{eq:4}.
\end{proof}

\begin{corollary}\label{cor.u}
  For all $N\ge 2$,
  \begin{equation*}
    U^+_N = \langle O_N, \bT^+_N\rangle.
  \end{equation*}  
\end{corollary}
\begin{proof}
  This follows from \Cref{th.u} and the fact that $U_N$ is topologically generated by $O_N$ and $\bT_N<\bT^+_N$. 
\end{proof}

As a consequence, we have the analogous orthogonal result:

\begin{theorem}\label{th.o}
  For all $N\ge 2$,
  \begin{equation}\label{eq:5}
    O^+_N = \langle O_N, T^+_N\rangle.
  \end{equation}  
\end{theorem}
\begin{proof}
  First, consider the projectivization $PO^+_N$ whose underlying $*$-algebra is generated by
  \begin{equation*}
    u^*_{ij}u_{kl} = u_{ij}u_{kl}
  \end{equation*}
  and the unitary analogue $PU^+_N$. The inclusion $O^+_N\to U^+_N$ induces an isomorphism
  \begin{equation*}
    PO^+_N\cong PU^+_N
  \end{equation*}
  (e.g. by \cite[Th\'eor\`eme 1]{ban-un} or \cite[Proposition 3.3]{chi-rfd}). Since \Cref{cor.u} ensures that
  \begin{equation}\label{eq:6}
    PO^+_N=PU^+_N=\langle PO_N,P\bT^+_N\rangle  = \langle PO_N,PT^+_N\rangle,
  \end{equation}
we at least know that \Cref{eq:5} holds ``projectively''. 

Now let
\begin{equation*}
  G=\langle O_N,T^+_N\rangle \le O^+_N. 
\end{equation*}
Since the center $\bZ_2<O^+_N$ is contained in $O_N$ and $T^+_N$, we have a commutative diagram
\begin{equation*}
 \begin{tikzpicture}[auto,baseline=(current  bounding  box.center)]
  \path[anchor=base] (0,0) node (1) {$\bC$} +(2,.5) node (lu) {$\cO(PO^+_N)$} +(4,.5) node (mu) {$\cO(O^+_N)$} +(6,0) node (r) {$\cO(\bZ_2)$} +(8,0) node (2) {$\bC$} +(2,-.5) node (ld) {$\cO(H)$} +(4,-.5) node (md) {$\cO(G)$};
  \draw[->] (1) to[bend left=6] node[pos=.5,auto] {$\scriptstyle $} (lu);
  \draw[->] (1) to[bend right=6] node[pos=.5,auto] {$\scriptstyle $} (ld);
  \draw[->] (lu) to[bend left=0] node[pos=.5,auto] {$\scriptstyle $} (mu);
  \draw[->] (ld) to[bend left=0] node[pos=.5,auto] {$\scriptstyle $} (md);
  \draw[->] (mu) to[bend left=6] node[pos=.5,auto] {$\scriptstyle $} (r);
  \draw[->] (md) to[bend right=6] node[pos=.5,auto] {$\scriptstyle $} (r);
  \draw[->] (r) to[bend right=0] node[pos=.5,auto] {$\scriptstyle $} (2);
  \draw[->] (lu) to[bend right=0] node[pos=.5,auto] {$\scriptstyle $} (ld);
  \draw[->] (mu) to[bend right=0] node[pos=.5,auto] {$\scriptstyle $} (md);
 \end{tikzpicture}
\end{equation*}
of Hopf algebras with exact rows, surjective columns, and
\begin{equation*}
  H = \langle PO_N,PT^+_N\rangle \le PO^+_N. 
\end{equation*}
Because $H=PO^+_N$ by \Cref{eq:6} the left hand vertical arrow is an isomorphism, and hence so is the right hand vertical arrow. But this means precisely that we have \Cref{eq:5}, as desired.
\end{proof}

\subsection{Dual free products by the circle}\label{subse.g}

The following result, used above in the proof of \Cref{th.u}, might be of some independent interest. 

\begin{theorem}\label{th.tg}
  Let $G$ be a compact connected Lie group and $T\le G$ a maximal torus. We then have
  \begin{equation*}
    \bT\,\hat *\, G = \langle \bT\,\hat *\, T,G\rangle.
  \end{equation*}
\end{theorem}
\begin{proof}
  We have to argue that the Hopf algebra surjections
  \begin{equation*}
    \cO(\bT)*\cO(G)\to \cO(\bT)*\cO(T),\quad  \cO(\bT)*\cO(G)\to \cO(G)
  \end{equation*}
  do not factor through any proper Hopf quotient of $\cO(\bT)*\cO(G)$. To that end, let
  \begin{equation*}
    \cO(\bT)\to \cO(G)\to H
  \end{equation*}
  be the smallest Hopf quotient factoring the two maps. Since all maximal tori of $G$ are mutually conjugate, we have factorizations
  \begin{equation*}
    \cO(\bT)*\cO(G)\to H\to \cO(\bT)*\cO(T_i)
  \end{equation*}
  for every maximal torus $T_i\le G$.


  Consider a finite set of maximal tori $T_i\le G$, $1\le i\le k$ such that the product
  \begin{equation*}
    T_1\times\cdots\times T_k\to G
  \end{equation*}
  is onto. At the level of function algebras, this means that the iterated coproduct
  \begin{equation}\label{eq:9}
    \Delta^{(k-1)}:\cO(G)\to \cO(T_1)\otimes\cdots\otimes \cO(T_k)
  \end{equation}
  is an embedding. 

  The analogous iterated coproduct
  \begin{equation}\label{eq:8}
    \Delta^{(k-1)}:\cO(\bT)*\cO(G)\to (\cO(\bT)*\cO(T_1))\otimes\cdots\otimes (\cO(\bT)*\cO(T_k))
  \end{equation}
  lands inside the algebra $\cA$ generated by the diagonal subalgebra
  \begin{equation*}
    \cO(\bT)\subset \cO(\bT)^{\otimes k}
  \end{equation*}
  and
  \begin{equation*}
    \cO(T_1)\otimes\cdots\otimes \cO(T_k). 
  \end{equation*}
  Now note that
  \begin{equation*}
    \cA\cong \cO(\bT)*(\cO(T_1)\otimes\cdots\otimes \cO(T_k))
  \end{equation*}
  and hence \Cref{eq:8} must be one-to-one because \Cref{eq:9} is. Since on the other hand \Cref{eq:8} factors through the quotient $\cO(\bT)*\cO(G)\to H$, it follows, as desired, that this quotient is the identity.
\end{proof}



\begin{thebibliography}{10}

\bibitem{abe}
Eiichi Abe.
\newblock {\em Hopf algebras}, volume~74 of {\em Cambridge Tracts in
  Mathematics}.
\newblock Cambridge University Press, Cambridge-New York, 1980.
\newblock Translated from the Japanese by Hisae Kinoshita and Hiroko Tanaka.

\bibitem{ban-hl}
T.~{Banica}.
\newblock {A note on the half-liberation operation}.
\newblock {\em arXiv e-prints}, March 2019.

\bibitem{ban-tor}
T.~{Banica}.
\newblock {Compact quantum groups generated by their tori}.
\newblock {\em arXiv e-prints}, March 2019.

\bibitem{ban-ax}
T.~{Banica}.
\newblock {Quantum groups under very strong axioms}.
\newblock {\em arXiv e-prints}, January 2019.

\bibitem{ban-un}
Teodor Banica.
\newblock Le groupe quantique compact libre {${\rm U}(n)$}.
\newblock {\em Comm. Math. Phys.}, 190(1):143--172, 1997.

\bibitem{ban-pat}
Teodor Banica and Issan Patri.
\newblock Maximal torus theory for compact quantum groups.
\newblock {\em Illinois J. Math.}, 61(1-2):151--170, 2017.

\bibitem{bbcw}
A.~{Bhattacharya}, M.~{Brannan}, A.~{Chirvasitu}, and S.~{Wang}.
\newblock {Kirchberg factorization and residual finiteness for discrete quantum
  groups}.
\newblock {\em ArXiv e-prints}, December 2017.

\bibitem{bw}
Angshuman Bhattacharya and Shuzhou Wang.
\newblock Kirchberg's factorization property for discrete quantum groups.
\newblock {\em Bull. Lond. Math. Soc.}, 48(5):866--876, 2016.

\bibitem{bcv}
Michael Brannan, Beno\^\i{t} Collins, and Roland Vergnioux.
\newblock The {C}onnes embedding property for quantum group von {N}eumann
  algebras.
\newblock {\em Trans. Amer. Math. Soc.}, 369(6):3799--3819, 2017.

\bibitem{chi-rfd}
Alexandru Chirvasitu.
\newblock Residually finite quantum group algebras.
\newblock {\em J. Funct. Anal.}, 268(11):3508--3533, 2015.

\bibitem{cep}
A.~Connes.
\newblock Classification of injective factors. {C}ases {$II_{1},$} {$II_{\infty
  },$} {$III_{\lambda },$} {$\lambda \not=1$}.
\newblock {\em Ann. of Math. (2)}, 104(1):73--115, 1976.

\bibitem{dk}
Mathijs~S. Dijkhuizen and Tom~H. Koornwinder.
\newblock C{QG} algebras: a direct algebraic approach to compact quantum
  groups.
\newblock {\em Lett. Math. Phys.}, 32(4):315--330, 1994.

\bibitem{egno}
Pavel Etingof, Shlomo Gelaki, Dmitri Nikshych, and Victor Ostrik.
\newblock {\em Tensor categories}, volume 205 of {\em Mathematical Surveys and
  Monographs}.
\newblock American Mathematical Society, Providence, RI, 2015.

\bibitem{kirch}
Eberhard Kirchberg.
\newblock Discrete groups with {K}azhdan's property {${\rm T}$} and
  factorization property are residually finite.
\newblock {\em Math. Ann.}, 299(3):551--563, 1994.

\bibitem{mon}
Susan Montgomery.
\newblock {\em Hopf algebras and their actions on rings}, volume~82 of {\em
  CBMS Regional Conference Series in Mathematics}.
\newblock Published for the Conference Board of the Mathematical Sciences,
  Washington, DC; by the American Mathematical Society, Providence, RI, 1993.

\bibitem{rad}
David~E. Radford.
\newblock {\em Hopf algebras}, volume~49 of {\em Series on Knots and
  Everything}.
\newblock World Scientific Publishing Co. Pte. Ltd., Hackensack, NJ, 2012.

\bibitem{swe}
Moss~E. Sweedler.
\newblock {\em Hopf algebras}.
\newblock Mathematics Lecture Note Series. W. A. Benjamin, Inc., New York,
  1969.

\bibitem{DaeWan96}
Alfons Van~Daele and Shuzhou Wang.
\newblock Universal quantum groups.
\newblock {\em Internat. J. Math.}, 7(2):255--263, 1996.

\bibitem{Wan98}
Shuzhou Wang.
\newblock Quantum symmetry groups of finite spaces.
\newblock {\em Comm. Math. Phys.}, 195(1):195--211, 1998.

\end{thebibliography}
\bibliographystyle{plain}
\addcontentsline{toc}{section}{References}

\def\polhk#1{\setbox0=\hbox{#1}{\ooalign{\hidewidth
  \lower1.5ex\hbox{`}\hidewidth\crcr\unhbox0}}}

\Addresses

\end{document}